\documentclass[11pt]{amsart}
 \usepackage{amsmath,amsthm,amsfonts,amssymb}
\usepackage{verbatim}
\usepackage{latexsym}
\usepackage{nicefrac}
\usepackage{color}
\usepackage{ulem}
\newcommand{\R}{{\mathbb{R}}}
\newcommand{\C}{{\mathbb{C}}}

\newcommand{\N}{\mathbb{ N}}
\newcommand{\Z}{\mathbb{ Z}}

\newtheorem{theorem}{Theorem}[section]
\newtheorem{lemma}[theorem]{Lemma}

\newtheorem{definition}[theorem]{Definition}
\newtheorem{corollary}[theorem]{Corollary}

\newtheorem{remark}[theorem]{Remark}
\newcommand{\ro}[1]{{#1}}

\begin{document}
\title{Embeddings of anisotropic Besov spaces into Sobolev spaces}
\author{David Bartusel}
\author{Hartmut F\"uhr}
\email{david.bartusel@matha.rwth-aachen.de}
\email{fuehr@matha.rwth-aachen.de}
\address{Lehrstuhl A f\"ur Mathematik, RWTH Aachen University, D-52056 Aachen, Germany}
\address{Lehrstuhl A f\"ur Mathematik, RWTH Aachen University, D-52056 Aachen, Germany}
% \author{Jahangir Cheshmavar, Hartmut F\"uhr\\
% \footnotesize\texttt{{fuehr@matha.rwth-aachen.de}} } 

%AMS Subject 22D10, 42C15, 46E35, 42C40
\maketitle

\begin{abstract}
We study the embeddings of (homogeneous and inhomogeneous) anisotropic Besov spaces associated to an expansive matrix $A$ into Sobolev spaces, with focus on the influence of $A$ on the embedding behaviour. For a large range of parameters, we derive sharp characterizations of embeddings. 
\end{abstract}



\section{Introduction}\label{introduction}

In this paper we investigate embeddings of Besov spaces associated to expansive matrices $A$ into isotropic Sobolev spaces. Recall that there exist two scales of Besov spaces for each expansive matrix $A$, letting $\dot{B}_{p,r}^\alpha(A)$ denote the \textit{homogeneous} Besov spaces, and $B_{p,r}^\alpha(A)$ the \textit{inhomogeneous} Besov spaces with integrability exponents $p,q$ and smoothness parameter $\alpha$  \cite{Bow05}. 

A systematic study of the embedding behaviour of \ro{isotropic} Besov spaces (corresponding to the matrix $A = 2 \cdot I_d$) into Sobolev spaces can be found in \cite{FV}; see also \cite{Triebel_FSI} for related results. To our knowledge, the current understanding of the relationship between anisotropic Besov spaces and Sobolev spaces is limited to fairly special settings, see e.g. \cite{MR2400793}. The paper \cite{MR1884234} studies embeddings of certain anisotropic Besov spaces in anisotropic Sobolev spaces. By contrast, we are not aware of a systematic investigation of the embedding behaviour of the anisotropic Besov spaces $B_{p,r}^\alpha(A)$ and $\dot{B}_{p,r}^\alpha(A)$ into isotropic Sobolev spaces. Note that a considerable challenge in this context comes from the great variety of scales of Besov spaces arising from different choices of anisotropic matrices $A$. Understanding the influence of these matrices on the embedding behaviour is the main motivation of this paper.

%While there \ro{are} still reasonable grounds to expect that elements of, say $B_{p,r}^\alpha(A)$ are increasingly smooth as $\alpha$ increases (since the associated norm punishes the contribution of large frequencies in the Fourier decomposition of $f$ more severely), the precise relationship of this anisotropic class of smoothness spaces to the more classical scale of isotropic Sobolev spaces is currently not well understood. 

\subsection{Overview of the paper}
We aim to derive necessary and sufficient conditions for the embedding statements
\[
  \dot{B}_{p,r}^\alpha(A) \hookrightarrow W^{\ro{n,q}} \mbox{ , }
  {B}_{p,r}^\alpha(A) \hookrightarrow W^{\ro{n,q}}~.
\] 
Our paper rests on a description of anisotropic Besov spaces as decomposition spaces, established in \cite{MR4135424}, and general embedding theorems for decomposition spaces proved in \cite{FV}. We review the pertinent definitions and results in Section 2, and make some observations that will allow to reduce the discussion of general expansive matrices $A$ to certain normal forms. We then put the criteria from \cite{FV} to work, first on the homogeneous case in Section 3, and on the inhomogeneous case in Section 4. For both settings we derive necessary and sufficient criteria for embeddings, and discuss the sharpness of these criteria. It turns out that homogeneous Besov spaces show a significantly less diverse embedding behaviour than their inhomogeneous counterparts. In particular, their embedding behaviour is indifferent to the choice of the expansive matrix $A$. By contrast, the embedding behaviour of inhomogeneous Besov spaces does reflect properties of $A$.
\ro{In particular, a wider range of eigenvalues (i.e. larger anisotropy) causes more restrictive embedding properties.} Furthermore, the newly-defined notion of \textit{asymptotically norm diagonal} matrices allows a fairly sharp description of the extent to which the off-diagonal part of $A$ influences the embedding behaviour. Remark \ref{rem:summary} contains a short summary of this discussion.

\section{Preliminaries}

Throughout the paper, and unless fixed locally otherwise, we let $d\ro{{}\in\N}, n \in \mathbb{N}_{\ro{0}}$, $0< p,q,r \le \infty$ and $\alpha \in \mathbb{R}$. 

\subsection{Sobolev spaces}

Fix $1\leq q\leq\infty$. Given any $f\in L^q(\R^d)$, we let $\partial^{\alpha}f$ denote its distributional derivative. We let 
\begin{equation*}
W^{n,q}(\R^d)=\{f\in L^{q}(\R^d)~|~\partial^{\alpha}f\in L^q(\R^d)\text{ for all } \alpha\in\N_0^d\text{ with } |\alpha|\leq n\}
\end{equation*}
with the norm
\begin{equation*}
\|f\|_{W^{n,q}}=\sum\limits_{\substack{\alpha\in\N_0^d,\\|\alpha|\leq n}} \|\partial^{\alpha}f\|_{L^q}.
\end{equation*}

Given $q\in(0,1)$, we let 
\begin{equation*}
W^{n,q}(\R^d):=\overline{\left\{(\partial^{\alpha}f)_{\substack{\alpha\in\N_0^d,\\|\alpha|\leq n}}~\Bigg|~f\in C^{\infty}(\R^d)\text{ with } \partial^{\alpha}f\in L^q(\R^d)\text{ for all } \alpha\in\N_0^d\text{ with }|\alpha|\leq n\right\}}.
\end{equation*}
Here the closure is taken with respect to the product space norm 
\begin{equation*}
\prod\limits_{\substack{\alpha\in\N_0^d,\\|\alpha|\leq n}} L^q(\R^d)~\ro{.}
\end{equation*}

As noted in \cite{FV}, there are several distinct definitions of Sobolev spaces for exponents $q<1$. The definition used here follows \cite{MR374900}, and it necessitates a clarification of what we mean by an embedding theorem of anisotropic Besov spaces into Sobolev spaces, see Definition \ro{\ref{defn:embed}} below, and the remarks preceding and following it. 

Following \cite{FV}\ro{,} we define for $1\leq q\leq\infty$ the conjugate exponent $q^{\prime}$ as \ro{the unique solution of} $\frac{1}{q}+\frac{1}{q^{\prime}}=1$; for $q<1$ \ro{we} let $q^{\prime}:=\infty$. We shall also need 
\begin{equation*}
q^{\nabla}:=\min\{q,q^{\prime}\}=\begin{cases} q, &\text{if } q\leq 2,\\ q^{\prime},&\text{if } q>2.\end{cases}
\end{equation*}

We note an elementary lemma concerning the action of dilations on Sobolev spaces, which follows by (somewhat tedious) computations involving the chain rule:
\begin{lemma} \label{lem:compose_sobolev}
 Given an invertible matrix $C \in \mathbb{R}^{d \times d}$, the linear map $f \mapsto f \circ C$ defines a bounded, invertible operator on $W^{n,q}(\R^d)$. 
\end{lemma}

\subsection{Anisotropic Besov spaces and their decomposition space description}

Our exposition regarding anisotropic Besov spaces  follows \cite{Bow05}. Let us start with some preliminaries and basic notions. We will use the following normalization of the Fourier transform $\mathcal{F} : {\rm L}^1(\mathbb{R}^d) \to C_0(\mathbb{R}^d)$: For all $f \in {\rm L}^1(\mathbb{R}^d)$,
\[
 \mathcal{F}(f)(\xi) = \widehat{f}(\xi) = \int_{\mathbb{R}^d} f(x) e^{- 2 \pi i \langle \xi, x \rangle} dx~.
\]
$\mathcal{S}(\mathbb{R}^d)$ denotes the space of Schwartz functions, $\mathcal{S}'(\mathbb{R}^d)$ its dual, the space of tempered distributions. As is well-known, the Fourier transform extends canonically to $\mathcal{S}'(\mathbb{R}^d)$. We let $\mathcal{P}$ denote the space of polynomials on $\mathbb{R}^d$, which can be viewed as a subspace of $\mathcal{S}(\mathbb{R}^d)$. For these definitions and basic properties of the Fourier transform, we refer to \cite{Ru_FA}.

Given an open subset $\mathcal{O} \subset \mathbb{R}^d$, we let $\mathcal{D}(\mathcal{O}) = C_c^\infty(\mathcal{O})$, the space of smooth, compactly supported functions on $\mathcal{O}$, endowed with the usual topology \cite{Ru_FA}. We let $\mathcal{D}'(\mathcal{O})$ denote its dual space\ro{, the space of distributions}.
We use ${\rm supp}(f) = \overline{f^{-1}(\mathbb{C} \setminus \{  0 \})}$ for the support of a function $f$. Given a Borel subset $C \subset \mathbb{R}^d$, $\lambda(C)$ denotes its Lebesgue measure. The cardinality of a set $X$ is denoted by $|X|$.

Given a vector $x = (x_1,\ldots,x_d)^T \in \mathbb{R}^d$, we
denote by $|x| = \left( \sum_{i=1}^d |x_i|^2 \right)^{1/2}$ its
euclidean length. Given a matrix $A \in \mathbb{R}^d$, we let $\|
A \| = \sup_{|x| = 1} |A x|$.

The definition of anisotropic Besov spaces is based on the notion of expansive matrices.
\begin{definition}
 \label{defn:expansive}
 A matrix $A \in {\rm GL}(d,\mathbb{R})$ is called {\bf expansive} if all its (possibly complex) eigenvalues $\lambda$ fulfill $|\lambda|>1$.
\end{definition}

\begin{definition}
 Let $A \in {\rm GL}(d,\mathbb{R})$ be an expansive matrix. $\psi \in \mathcal{S}(\mathbb{R}^d)$ is called {\bf $A$-wavelet} if $\widehat{\psi}$ is compactly supported away from $0$, and fulfills 
 \begin{eqnarray}
%  \label{eqn:def_wv1}  & & {\rm supp}(\widehat{\psi}) \subset [-1,1]^d \setminus \{ 0 \}~, \\
  \label{eqn:def_wv2}  & & \forall \xi \in \mathbb{R}^d \setminus \{0 \}~:~\sum_{j \in \mathbb{Z}} \left| \widehat{\psi}((A^T)^{-j} \xi) \right| = 1~.
 \end{eqnarray}
%  It is called a {\bf $A$-wavelet of the second kind} if instead of (\ref{eqn:def_wv2}) it fulfills
%  \begin{equation}
%   \label{eqn:def_wv3} 
% \forall \xi \in \mathbb{R}^d \setminus \{0 \}~:~\sum_{j \in \mathbb{Z}}  \widehat{\psi}((A^T)^{-j} \xi)  = 1~.
%  \end{equation} 

 Given a wavelet $\psi$, we define $\psi_j(x) = |{\rm det}(A)|^j \psi(A^jx)$, for $j \in \mathbb{Z}$.
 Given a wavelet $\psi$, a function $\psi_0 \in S(\mathbb{R})$ is called {\bf low-pass complement to $\psi$}, if $\widehat{\psi}_{\ro{0}}$ is compactly supported, with
 \begin{equation}
  \forall  \xi \in \mathbb{R}^d ~:~ |\widehat{\psi}_0(\xi)|^2 + \sum_{j \in \ro{\mathbb{N}}} |\widehat{\psi}((A^T)^{-j}\xi| = 1~.
 \end{equation}
 The inhomogeneous wavelet system $(\psi_j^i)_{j \in \mathbb{N}_0}$ is defined by $\psi_j^i = \psi_j$, for $j \ge 1$, and $\psi_0^i = \psi_0$.
\end{definition}

\begin{definition}
 Let $A$ denote an expansive matrix, $\alpha \in \mathbb{R}$, and $0 < r \le \infty$.  
 For $\alpha \in \mathbb{R}$, define the weight
 \[
  v_{\alpha,A} : \mathbb{Z} \to \mathbb{R}^+~,~v_{\alpha,A}(j) = |{\rm det}(A)|^{j \alpha}~.
 \]
 The space $\ell^r_{v_{\alpha,A}}(\mathbb{Z})$ is the space of sequences $(c_j)_{j \in \mathbb{Z}}$ with the property that $(\ro{v_{\alpha,A}(j)} c_j)_{j \in \mathbb{Z}} \in \ell^r\ro{(\Z)}$, endowed with the obvious (quasi-)norm.  The space $\ell^r_{v_{\alpha,A}}(\mathbb{N}_0)$ is defined analogously. Since the precise meaning can usually be inferred from the context, we will typically write $\ell^r_{v_{\alpha,A}}$ for either of the two spaces.
\end{definition}

\begin{definition}
 \label{defn:an_bes}
 Let $\alpha \in \mathbb{R}$, $0 < p,r \le \infty$.  Let $A$ be an expansive matrix, and $\psi$ an $A$-wavelet, with low-pass complement $\psi_0$.
 \begin{enumerate}
  \item[(a)]
 We define the {\bf anisotropic homogeneous Besov (quasi-) norm} by letting, for
 given $f \in \mathcal{S}'(\mathbb{R}^d)$,
 \begin{equation} \label{eqn:def_bnorm}
  \| f \|_{\dot{B}_{p,r}^\alpha(A)} = \left\| \left( \left\| f \ast \psi_j\right\|_p \right)_{j \in \mathbb{Z}}
  \right\|_{\ell^r_{v_{\alpha,A}}}.
 \end{equation}
 Here we use $\| \cdot \|_p$ to denote the (quasi-)norm of the space ${\rm L}^p(\mathbb{R}^d)$.
 We let  $\dot{B}_{p,r}^\alpha(A)$ denote the space of all tempered distributions $f$ with $\| f \|_{\dot{B}_{p,r}^\alpha(A)} < \infty$. We identify elements of $\dot{B}_{p,r}^\alpha(A)$ that only differ by a polynomial.
 \item[(b)] The {\bf anisotropic inhomogeneous Besov (quasi-) norm} is defined for
 given\linebreak $f \in \mathcal{S}'(\mathbb{R}^d)$ by
 \begin{equation} \label{eqn:def_bnorm_ih}
  \| f \|_{{B}_{p,r}^\alpha(A)} = \left\| \left( \left\| f \ast \psi_j^i\right\|_p \right)_{j \in \mathbb{N}_0}
  \right\|_{\ell^r_{v_{\alpha,A}}}.
 \end{equation}
 We let $B_{p,r}^\alpha(A)$ denote the space of all tempered distributions $f$ with $\| f \|_{B_{p,r}^\alpha(A)} < \infty$.
 \end{enumerate}
\end{definition}

The following remark collects some fundamental facts concerning anisotropic Besov spaces, established in \cite{Bow05}.
\begin{remark}
It is crucial for the well-definedness of the theory, that $A$-wavelets actually exist, for any expansive matrix, and that different choices of wavelets result in equivalent Besov \linebreak(quasi-)norms, and identical Be\ro{so}v spaces.  Furthermore, anisotropic Besov spaces are quasi-Banach spaces, and Banach spaces iff $1 \le p, r \le \infty$ holds. 
\end{remark}

Our general approach hinges on an alternative description of anisotropic Besov spaces, namely as so-called \textbf{decomposition spaces}. These scales of spaces are based on suitable coverings of the frequency domain associated to an expansive matrix. 
% 
% 
% In order to identify anisotropic Besov spaces as decomposition spaces, we need to associate coverings to expansive matrices. 

\begin{definition} \label{defn:induced_cover}
 Let $A$ denote an expansive matrix. Let $C \subset \mathbb{R}^d$ be open, such that $\overline{C}$ is a compact subset of $\mathbb{R}^d \setminus \{ 0 \}$, and define, for $j \in \mathbb{Z}$,
 \[
  Q_j = A^j \overline{C}~.
 \]
 If $\bigcup_{j \in \mathbb{Z}} Q_j = \mathbb{R}^d \setminus \{ 0 \}$,
 $\mathcal{Q} = (Q_j)_{j \in \mathbb{Z}}$ is called {\bf homogeneous covering induced by $A$}. An {\bf inhomogeneous covering induced by $A$} is given by the family $\mathcal{Q}_A^i = (Q_j^i)_{j \in \mathbb{N}_0}$, where\linebreak $Q_j^i = Q_j = A^j \overline{C}$ for $j \ge 1$, and $Q_0^i = \overline{C_0}$, for a relatively compact open set $C_0$ with the property that
 \[
  \bigcup_{j \in \mathbb{N}_0} Q_j^i = \mathbb{R}^d~.
 \]
 \end{definition}

\ro{Given an $A$-wavelet $\psi$ with low-pass complement $\psi_0$, common choices in Definition \ref{defn:induced_cover} are $C:={\hat{\psi}}^{-1}(\C\setminus\{0\})$ and $C_0:={\hat{\psi}}_0^{-1}(\C\setminus\{0\})$ (i.e. $\overline{C}=\operatorname{supp}(\hat{\psi})$ and $\overline{C_0}=\operatorname{supp}(\hat{\psi}_0)$). In fact, this always yields the desired covering properties.}
 
 The next result formulates the identification of Besov spaces as decomposition spaces.

 \begin{theorem} \label{thm:besov_as_decsp}
 Let $A$ denote an expansive matrix, and let $\mathcal{Q}_A$ denote a homogeneous covering induced by $A^T$. Then, up to suitable identification, one has 
 \[
 \dot{B}_{p,r}^\alpha(A) = \mathcal{D}(\mathcal{Q}_A,{\rm L}^p,\ell^r_{v_{\alpha,A}})~,
 \]
 with equivalent norms. 
 
 Similarly, if $\mathcal{Q}_A^i$ denotes an inhomogeneous covering induced by $A^T$, then
  \[
  {B}_{p,r}^\alpha(A) = \mathcal{D}(\mathcal{Q}_A^i,{\rm
 L}^p,\ell^r_{v_{\alpha,A}})~,
 \]
up to suitable identification, and with equivalent norms. Here  $v_{\alpha,A}$ denotes the restriction of the weight for the homogeneous setting to $\mathbb{N}_0$.
 \end{theorem}
 
 \begin{proof}
We shortly sketch the arguments and identifications necessary to establish these equalities, and refer to \cite{MR4135424} for the technical details. We only treat the homogeneous case, the argument for the inhomogeneous case is entirely analogous, and in fact somewhat easier. 

We let $\mathcal{O} = \mathbb{R}^d$. In the following, we shortly recap the definition of \textbf{space side decomposition spaces}, as employed in \cite{FV}.\\ We start out by considering the dual covering $\mathcal{Q}_A = (Q_j)_{j \in \mathbb{Z}} = \left( (A^T)^j C \right)_{j \in \mathbb{Z}}$, where $C \subset \mathcal{O}$ is an open set with compact closure in $\mathcal{O}$\ro{, such that $\bigcup_{j\in\Z} Q_j =\mathcal{O}$.} Since $A$ is expansive, these properties imply the existence of $r >0$ such that 
\begin{equation} \label{eqn:nonempty_is}
 Q_j \cap Q_\ell = \emptyset \mbox{ if } |j-\ell| \ge r~.
\end{equation} Using these properties, one easily verifies that $\mathcal{Q}_A$ is a \textbf{\ro{tight, regular semi-structured} covering} in the sense of \cite{FV} with matrices $T_j = (A^T)^j$ and vectors $b_j = 0$\ro{, and thus an $L^p$-decomposition covering.} In addition, (\ref{eqn:nonempty_is}) implies that the weight $v_{\alpha,A}$ is $\mathcal{Q}_A$-moderate in the sense of \cite[Definition 2.9]{FV}. 

Picking an $A$\ro{-}wavelet $\psi$ with support $\widehat{\psi} \subset C$, one easily verifies that $\varphi_j = \widehat{\psi}_j$ defines an \textbf{$L^p$-BAPU subordinate to $\mathcal{Q}_A$} in the sense of \cite[Definition 2.3]{FV}.

The final ingredient of the definition of the decomposition space $ \mathcal{D}(\mathcal{Q}_A,{\rm L}^p,\ell^r_{v_{\alpha,A}})$ is the associated \textbf{reservoir space}. 
We define 
\[
 \ro{\mathcal{Z}(\mathcal{O}) = \{\mathcal{F}(f)~:~ f\in\mathcal{D}(\mathcal{O}) \}~,}
\] as the \ro{image} under Fourier transform of the space of test functions on $\mathcal{O}$, given by\linebreak $\mathcal{D}(\mathcal{O}) = C_c^\infty(\mathcal{O})$, endowed with the standard topology. We \ro{endow $\mathcal{Z}(\mathcal{O})$ with the unique topology that makes $\mathcal{F}$ a homeomorphism}, and let $\mathcal{Z}'(\mathcal{O})$ denote the associated dual space of continuous linear functionals on $\mathcal{Z}(\mathcal{O})$. There is also a well-defined Fourier transform on the dual spaces; more precisely, we have 
\[
 \mathcal{F} : \mathcal{Z}'(\mathcal{O}) \to \mathcal{D}'(\mathcal{O})~,~f \mapsto f \circ \ro{\mathcal{F}}~.
\]
With these definitions in place, the space-side decomposition spaces are defined as
\[
 \mathcal{D}(\mathcal{Q}_A,{\rm L}^p,\ell^r_{v_{\alpha,A}}) = \left\{ f \in \mathcal{Z}'(\mathcal{O}) ~:~
 \left( \| \mathcal{F}^{-1}(\varphi_j \widehat{f}) \|_p \right)_{j \in \mathbb{Z}} \in \ell^r_{v_{\alpha,A}}~ \right\}
\] with the obvious choice of quasi-norm \cite[Definition 2.12]{FV}. Now a comparison of this definition with our Definition \ref{defn:an_bes} shows that the norms are identically defined, but on different base spaces, and pending proper interpretations of the Fourier transforms.
The remaining steps in the identification therefore consist in establishing the following: 
\begin{enumerate}
 \item Every $f \in \mathcal{S}'(\mathbb{R}^d)$ satisfying $\| f \|_{\dot{B}_{p,r}^\alpha} <  \infty$ induces an element $\tilde{f} \in \mathcal{Z}'(\mathcal{O})$, and $\mathcal{F}(\tilde{f})$ viewed as element of $\mathcal{D}'(\mathcal{O})$ coincides with the Fourier transform of $f$ as tempered distribution. 
 \item Conversely, every $f \in \mathcal{Z}'(\mathcal{O})$ satisfying $\| f \|_{\mathcal{D}(\mathcal{Q}_A,{\rm L}^p,\ell^r_{v_{\alpha,A}})}< \infty$ induces a tempered distribution $\tilde{f}$, with $\mathcal{F}(f) \in \mathcal{D}'(\mathcal{O})$ coinciding with the Fourier transform of $\tilde{f}$ as tempered distribution.
\end{enumerate}
For the corroboration of these results we refer to the proof of Theorem 5.6 in \cite{MR4135424}.
\end{proof}

\begin{remark} \label{rem:embed}
We next give a formal definition of embeddings of anisotropic Besov spaces into Sobolev spaces. Informally, the elements of the Sobolev space $W^{n,q}$ are functions with the properties that their derivatives of order up to $n$ are $q$-integrable.

Hence for $1 \le  q \le \infty$ the canonical inclusion $L^q(\mathbb{R}^d) \subset \mathcal{S}'(\mathbb{R}^d)$ gives a well-defined meaning to the statement $f \in L^q(\mathbb{R}^d)$, for a tempered distribution $f$. 
Furthermore, for $1 \le q < \infty$, the only polynomial contained in $L^q(\R^d)$ is the zero polynomial, which implies that the statement $f \in L^q(\R^d)$ is even well-defined modulo polynomials.

The consequence of these observations is that the reservoir spaces used in the definition of homogeneous and inhomogeneous anisotropic Besov spaces allow to define the embedding statements 
\[
 \dot{B}_{p,r}^\alpha(A) \hookrightarrow W^{\ro{n,q}}~\mbox{ resp. }~ B_{p,r}^\alpha(A) \hookrightarrow W^{\ro{n,q}}
\] as simple inclusions, whenever $1 \le q \le \infty$.

Note however that this point of view is not available only for $0<q<1$, since $L^q(\mathbb{R}^d)$ has no canonical embedding into $S'(\mathbb{R}^d)$. 
Therefore, this case requires a somewhat different approach. The following definition allows a unified treatment of all $0<q\le \infty$. 
\end{remark}
 
\begin{definition}
 \label{defn:embed}
 Let $A$ denote an expansive matrix, and let $\psi$ denote an $A$-wavelet. We write 
 \[
  \dot{B}_{p,r}^\alpha(A) \hookrightarrow W^{\ro{n,q}}
 \] if for all $f \in \dot{B}_{p,r}^\alpha(A)$ and all multiindices $\beta \in \mathbb{N}_0^d$ satisfying $|\beta|\le n$ the operator
 \[
  \partial_\ast^\beta : \dot{B}_{p,r}^\alpha(A) \to {\rm L}^q(\mathbb{R}^d)~,~\partial_\ast^\beta(f) = \sum_{j \in \mathbb{Z}} \partial^\beta(f \ast \psi_j)
 \] is well-defined and bounded, with unconditional (quasi-)norm convergence of the sum on the right hand side. The embedding
 \[
   B_{p,r}^\alpha(A) \hookrightarrow W^{\ro{n,q}}
 \] is defined analogously.
\end{definition}

\ro{Whenever the embedding holds, we obtain canonical linear maps  $f \mapsto \partial_{\ast}^{0} f\in W^{n,q}(\R^d)$ in the case $q\geq 1$, and $f \mapsto (\partial_{\ast}^{\beta}f)_{|\beta|\leq n}\in W^{n,q}(\R^d)$ in the case $0<q<1$, which can be viewed as concrete implementations of the embedding. Moreover, note that the partial derivatives in the series expansions on the right-hand side may be taken in a classical sense, since $f\ast\psi_j$ is a smooth function.}
 
As a further application of the decomposition space description of anisotropic Besov spaces, we cite a classification result characterizing the influence of the expansive matrix on the Besov spaces. This result will allow important structural assumptions on $A$ that signficantly simplify our reasoning. 

\begin{definition}
 A matrix $A$ is called \textbf{in expansive normal form} if every eigenvalue of $A$ is $>1$ (in particular real), and ${\rm det}(A) = 2$. 
 
 We call a matrix \textbf{in expansive Jordan normal form} if it is simultaneously in Jordan normal form and \ro{in} expansive normal form. 
\end{definition}

Note that any matrix $A'$ that is conjugate to a matrix $A$ in expansive normal form is itself in expansive normal form. In particular, if $A'$ is the Jordan normal form of a matrix $A$ in expansive normal form, than $A'$ is in expansive \ro{Jordan} normal form. 

The main purpose of the expansive normal form of a matrix is the following result classifying expansive matrices with respect to the associated homogeneous anisotropic spaces. It combines Lemma 7.7, Theorem 7.9 and Remark 7.11. of \cite{MR4135424}. 
\begin{theorem}\label{thm:expnf}
For every expansive matrix $A$ there exists a unique matrix $B$ in expansive normal form such that  $\dot{B}_{p,r}^\alpha(A) = \dot{B}_{p,r}^\alpha (B)$ holds, for all $0 \ro{{}< p,r} \le \infty$ \ro{and $\alpha\in\R$}. This also implies \ro{$B_{p,r}^{\alpha}(A)=B_{p,r}^{\alpha}(B)$}.
\end{theorem}

The following lemma studies the composition action of invertible matrices on elements from an anisotropic Besov space. Note that this action does not necessarily leave such a space invariant. 

\begin{lemma} \label{lem:compose_Besov}
 Let $A \in \R^{d \times d}$ denote an expansive matrix and $C \in \R^{d \times d}$ be invertible. Then the composition map $f \mapsto f \circ C$ defines isomorphisms $\dot{B}_{p,r}^\alpha(CAC^{-1}) \to \dot{B}_{p,r}^\alpha(A)$ and\linebreak $B_{p,r}^\alpha(CAC^{-1}) \to B_{p,r}^\alpha(A)$.
\end{lemma}

\begin{proof}
 We shortly sketch the argument. We first note that the dilation operator $D_C(f) = f \circ D$, initially defined on functions, extends naturally to tempered distributions. 
 
 In order to verify the isomorphism property  $\dot{B}_{p,r}^\alpha(CAC^{-1}) \to \dot{B}_{p,r}^\alpha(A)$, we fix an $A$-wavelet $\psi$. Then a quick calculation using the behaviour of the Fourier transform under dilation establishes that $\psi \circ C^{-1}$ is a $CAC^{-1}$-wavelet. The associated family of dilated wavelets is then given by
 \[
  \psi_{j,C} = |{\rm det}(CAC^{-1})|^j (\psi \circ C^{-1}) \circ (CAC^{-1})^j = |{\rm det}(A)|^j \psi \circ A^j \circ C^{-1} = \psi_j \circ C^{-1} 
 \]

 Furthermore, one easily verifies for $f \in S'(\mathbb{R}^d))$ and $\varphi \in \mathcal{S}(\mathbb{R}^d)$ that 
 \[
  (f \circ C)\ast \varphi = |{\rm det}(C)|^{-1} \left( f \ast (\varphi \circ C^{-1}) \right) \circ C~. 
 \]
 
This entails
\begin{eqnarray*}
 \| (f \circ C) \ast \psi_j \|_p & = &    |{\rm det}(C)|^{-1} \left\| (f \ast ( \psi_j \circ C^{-1} ))\circ C \right\|_p \\ & = & |{\rm det}(C)|^{-1-1/p} \left\| f \ast (\psi_j \circ C^{-1}) \right\|_p = |{\rm det}(C)|^{-1-1/p}  \left\| f \ast \psi_{j,C} \right\|_p~.
\end{eqnarray*}
Taking weighted $\ell^r$ norms on both sides gives $\| f \circ C \|_{\dot{B}_{p,r}^\alpha(A)} =  |{\rm det}(C)|^{-1-1/p}
\| f \|_{\dot{B}_{p,r}^\alpha(CAC^{-1})}$, which is the desired result. 

The proof for inhomogeneous spaces is completely analogous.  
\end{proof}

Combining the behaviours of the respective spaces under dilation with the properties of the expansive normal form allows to reduce the general embedding problem to a class of special cases, namely to matrices that are simultaneously in Jordan and expansive normal form. 
\begin{corollary}
 Let $A$ denote an expansive matrix, $A'$ the expansive normal form of $A$ \ro{(from Theorem \ref{thm:expnf})}, and $A''$ the Jordan normal form of $A'$. Then 
 \begin{equation} \dot{B}_{p,r}^\alpha (A) \hookrightarrow W^{n,q}(\mathbb{R}^d) ~\Longleftrightarrow~ \dot{B}_{p,r}^\alpha  (A'') \hookrightarrow W^{n,q}(\mathbb{R}^d)~.
 \end{equation}
 The analogous equivalence holds for the inhomogeneous Besov spaces.  
\end{corollary}

\ro{For the remainder of the paper, we will thus assume that the generating matrix $A$ is in expansive Jordan normal form.}

The following theorem spells out the main result of \cite{FV} for the homogeneous and inhomogeneous coverings induced by (the transpose of) an expansive matrix, and thus prepares our subsequent characterizations of the embeddings. The theorem translates the embedding problem to a summability property of a suitably defined sequence, that is associated to the dual covering.

\begin{theorem} \label{thm:FV_embed_Besov}
 Given an expansive matrix $A$ and $\ro{0<t\leq\infty}$, define the sequence \linebreak$a^{(t)} = (a_j^{(t)})_{j \in \mathbb{Z}}$ of positive real numbers given by 
 \begin{equation}
  a_j^{(t)} = |{\rm det}(A)|^{j(1/p-1/t-\alpha)} \cdot (1+\| A^j \|^n) ~;
 \end{equation} for $t=\infty$ we use the convention $1/\infty = 0$. We let $a_+^{(t)} = (a_j^{(t)})_{j \ge 1}$. 

 \begin{enumerate}
  \item[(a)] Assume that $p \le q$ as well as $a^{(q)} \in \ell^{q^{\nabla}\cdot(\nicefrac{r}{q^{\nabla}})^{\prime}}(\mathbb{Z})$. Then 
  \begin{equation} \label{eqn:embed_Bdot_FV} \dot{B}_{p,r}^\alpha(A) \hookrightarrow W^{n,q}(\mathbb{R}^d)~.
   \end{equation}
 Conversely, (\ref{eqn:embed_Bdot_FV}) entails $p \le q$, as well as $a^{(q)} \in \ell^{q\cdot (\nicefrac{r}{q})'}(\mathbb{Z})$. Additional necessary conditions for (\ref{eqn:embed_Bdot_FV}) are
 \[ a^{(q)} \in \ell^{r'}(\mathbb{Z}) \mbox{ for } q=\infty~ \]
 and
 \[ 
  a^{(p)}_+ \in \ell^{2 \cdot(r/2)'}(\mathbb{N}_0)\ro{\text{ for } q<\infty}.  
 \]
 \item[(b)] Assume that $p \le q$ as well as $a_+^{(q)} \in \ell^{q^{\nabla}\cdot(\nicefrac{r}{q^{\nabla}})^{\prime}}(\mathbb{N}_{\ro{0}})$. Then 
  \begin{equation} \label{eqn:embed_B_FV} B_{p,r}^\alpha(A) \hookrightarrow W^{n,q}(\mathbb{R}^d)~.
   \end{equation}
   Conversely, (\ref{eqn:embed_B_FV}) entails $p \le q$ as well as $a_+^{(q)} \in \ell^{q\cdot (\nicefrac{r}{q})'}(\mathbb{N})$.
   Additional necessary conditions for (\ref{eqn:embed_B_FV}) are
	\[
	 \ro{a_+^{(q)} \in \ell^{r^{\prime}}(\N_0)\text{ for } q=\infty~,}
	 \]
	\[
	a_+^{(p)} \in \ell^{2\cdot (r/2)'}(\mathbb{N}_{\ro{0}}) \mbox{ for } q \ro{{}<{}} \infty~,
	\]
	and 
   	\[
    	a_+^{(2)} \in \ell^{2\cdot (r/2)'}(\mathbb{N}_{\ro{0}}) \mbox{ for } q \in [2,\infty)~.
   	\]
 \end{enumerate}
\end{theorem}

\begin{proof}
 We recall the identifications from Theorem  \ref{thm:besov_as_decsp}, 
 \[
   \dot{B}_{p,r}^\alpha(A) = \mathcal{D}(\mathcal{Q}_A,{\rm L}^p,\ell^r_{v_{\alpha,A}})
 \]
 and 
 \[
  {B}_{p,r}^\alpha(A) = \mathcal{D}(\mathcal{Q}_A^i,{\rm
 L}^p,\ell^r_{v_{\alpha,A}})
 \] based on the homogeneous and inhomogeneous coverings $\mathcal{Q}_A$ and $\mathcal{Q}_A^i$, induced by $A^T$, respectively\ro{, i.e.} $\mathcal{Q}_A = ((A^T)^j \overline{C})_{j \in \mathbb{Z}}$, as well as
 $\mathcal{Q}_A^i = ((A^T)^j \overline{C})_{j \in \mathbb{N}} \cup \{ \overline{C_0} \}$.

 We intend to apply \cite[Corollary 5.2]{FV} with 
 \[
  u(j) = v_{\alpha,A}(j) =  |{\rm det}(A)|^{j \alpha}~.
 \] The covering-dependent weight $w^{(t)}$ from the cited result is computed as
 \[
  w^{(t)}(j) = |{\rm det}(A^j)|^{\frac{1}{p}-\frac{1}{t}} (1+\| A^j \|^{\ro{n}})~,
 \] where we used that both determinant and norms are invariant under transposition. With these objects, and some straightforward simplifications, the conditions of \cite[Corollary 5.2]{FV} are seen to specialize to the lists of conditions in (a) and (b).
 
 Note in particular that the conditions involving sequences of the type $a_+^{(t)}$, which occur both in parts (a) and (b), can be derived from  \cite[Corollary 5.2\ro{c)}]{FV} by observing that \linebreak$\sup_{j \ro{\geq} 0} \| A^{\ro{-j}} \| < \infty$. 
\end{proof}

\begin{remark}\label{rem:equiv}
 As already noted, $0<q\le 2$ entails $q = q^\nabla$, and the necessary and sufficient conditions of the theorem coincide.
 \ro{Additionally, for $q=\infty$, we obtain $q^{\nabla}\cdot\left(\nicefrac{r}{q^{\nabla}}\right)^{\prime}=r^{\prime}$, which yields equivalent characterizations as well.}
 Further comments concerning the gaps between these conditions will be found below. 
\end{remark}

\section{Embeddings of homogeneous Besov spaces}

Throughout the remainder of this paper, and unless otherwise specified, $A$ will always refer to an expansive matrix, and we shall use $0 < p,r,q \le \infty$, $n \in \mathbb{N}_0$ and $\alpha \in \mathbb{R}$. Furthermore, we will use 
\begin{equation} \label{eqn:def_nst}
 \ro{n^* =\alpha+\frac{1}{q}-\frac{1}{p}~,}
\end{equation} which turns out to be a very useful quantity, as explained in the next remark

\begin{remark} \label{rem:recurrent_themes}
\ro{Many of the} necessary and sufficient conditions of Theorem \ref{thm:FV_embed_Besov} for the embedding statements 
\[ \dot{B}_{p,r}^\alpha(A) \hookrightarrow W^{n,q}(\mathbb{R}^d)~,~ B_{p,r}^\alpha(A) \hookrightarrow W^{n,q}(\mathbb{R}^d)
 \]
\ro{involve the sequences $a^{(q)}$ or $a_+^{(q)}$ respectively, whose elements are given by}
\[
 \ro{|{\rm det}(A)|^{j(1/p-1/q-\alpha)} \cdot (1+\| A^j \|^n)  = |{\rm det}(A)|^{-jn^*} + |{\rm det}(A)|^{-jn^*} \| A^j \|^n}
\] \ro{in terms of $n^{\ast}$}. Clearly, \ro{summability of this sequence (to a given exponent)} is equivalent to separate summability of the two terms on the right. 

Since $A$ is expansive, both $|{\rm det}(A)|^j$ and $\| A^j \|$ grow exponentially as $j \to \infty$, and decay exponentially for $j \to -\infty$. In particular, $a^{(q)}$ is bounded iff it is constant, and this holds iff $n=n^* = 0$. 

\ro{For $a_+^{(q)}$ (i.e. in the inhomogeneous case),} summability requires a compensation between the two factors, \ro{which will allow us to establish a relationship between $n$ and $n^{\ast}$. This will be discussed in Section 4.} %\ro{\sout{In addition, finiteness of the sum over $j \in \mathbb{Z}$ requires a rather precise compensation, leading to more stringent conditions in the homogeneous case.}}
\end{remark}

The following theorem summarizes the embedding statements for the homogeneous case:
\begin{theorem} \label{thm:embed_hom}
 \begin{enumerate}
  \item[(a)] Assume that 
  \[ \dot{B}_{p,r}^\alpha(A) \hookrightarrow W^{n,q}(\mathbb{R}^d)~
   \] holds. Then the following relations exist between the various parameters:
   \begin{enumerate}
   \item[(i)]  $n=n^*=0$;
   \item[(ii)]$p \le q$ and $r \le q$;
   \item[(iii)] If $q= \infty$, then $r \le 1$;
   \item[(iv)] If $p=q$, then $r \le 2$. 
   \end{enumerate}   
  \item[(b)] Assume that the following conditions are fulfilled: 
  \begin{enumerate}
  \item[(i)] $n=n^* = 0$;
  \item[(ii)] $p\le q$ and $r \le q^\nabla$;
  \end{enumerate}
  Then the embedding 
  \[ \dot{B}_{p,r}^\alpha(A) \hookrightarrow W^{n,q}(\mathbb{R}^d)~
   \] holds.
 \end{enumerate}

\end{theorem}
\begin{proof}
 For the proof of (a), we first recall $p \le q$ as a necessary condition from Theorem \ref{thm:FV_embed_Besov}. We further have the necessary condition $a^{(q)} \in \ell^{q\cdot (\nicefrac{r}{q})'}(\mathbb{Z})$, which in particular requires boundedness of the sequence. But we have already noted in Remark \ref{rem:recurrent_themes} that boundedness is equivalent to $n=n^*=0$.
 
 The conditions derived so far entail that the sequence $a^{(q)}$ is in fact constant (and nonzero), and 
 under the condition $r>q$ we have $q (\nicefrac{r}{q})' < \infty$, contradicting $a^{(q)} \in \ell^{q\cdot (\nicefrac{r}{q})'}(\mathbb{Z})$. Thus $r \le q$ follows, and (i) and (ii) are shown. 

 If $q=\infty$, Theorem \ref{thm:FV_embed_Besov} notes the additional condition  $a^{(q)} \in \ell^{r'}(\mathbb{Z})$. Since we already know that this sequence is constant, this implies $r'= \infty$, or equivalently, $r \le 1$. 
 
 The remaining case for the proof of part (iv) is $p=q < \infty$. 
 Here we recall the necessary condition $a^{(p)}_+ \in \ell^{2 \cdot(r/2)'}(\mathbb{N}_0)$ from Theorem \ref{thm:FV_embed_Besov}(a). Under the assumptions $p=q$ and \ro{(by (i))} $n=n^*=0$, we have that $a^{(p)}_+$ is a nonzero, constant sequence. This implies $2 \cdot(r/2)' = \infty$, i.e. $r \le 2$. 

Finally, under the assumptions of (b), the sequence $a^{(q)}$ is constant, and $q^{\nabla}\cdot(\nicefrac{r}{q^{\nabla}}) = \infty$. Thus the sufficiency statement from Theorem \ref{thm:FV_embed_Besov} (a) applies. 
\end{proof}

\begin{remark}
Neither the sufficient nor the necessary conditions make any distinction between different expansive matrices, even though the associated scales of anisotropic Besov spaces are distinct whenever the associated expansive normal forms differ. 

This is not a byproduct of our approach to the characterization embeddings; it rather reflects the actual embedding behaviour of the involved spaces. Observe that the Theorem provides a sharp characterization for $q \in (0,2] \cup \{ \infty \}$. 
\end{remark}

\section{Embeddings of inhomogeneous Besov spaces}

The case of inhomogeneous Besov spaces is more complex, with a significant dependence on the matrix $A$. 
The following condition on $A$ will play an important role. 

\begin{definition}
 Let $A$ denote a matrix in expansive Jordan normal form, with maximal eigenvalue 
 \[
  \lambda_{max} = \max \{ \lambda~: \lambda \mbox{ eigenvalue of } A \}~. 
 \] We call $A$ \textbf{asymptotically norm diagonal} if the algebraic and geometric multiplicities of $\lambda_{max}$ coincide.
\end{definition}

\begin{remark} \label{rem:AND}
 The nomenclature of the previous definition is motivated by the observation that for any AND matrix $A$ in expansive Jordan normal form, and $j$ sufficiently large, \ro{one} has \linebreak$\| A^j \| = \lambda_{\max}^j$. To see this, we view $A$ as block diagonal matrix with $A_1$ denoting the Jordan block associated to $\lambda_{\max}$, and $A_2$ the remainder.\\
 We then note $\| A_2^j \| < \|A_1^j\ro{\|}$ and thus $\| A^j \| = \| A_1^j \| = \lambda_{\max}^j$ for sufficiently large $j$. Hence the norm of $A^j$ coincides with the norm of $B^j$, where $B$ is the diagonal matrix with the same eigenvalues as $A$. 
\end{remark}

We first note \ro{the} necessary conditions for the embedding\ro{, most of which} are independent of the AND property.
% \begin{theorem} \label{thm:embed_inhom_nec_i}
% Assume that 
%  \[ B_{p,r}^\alpha(A) \hookrightarrow W^{n,q}(\mathbb{R}^d)~
%   \] holds, and $A$ is in expansive Jordan normal form. 
% Then the following relations exist between the various parameters:
% \begin{enumerate}
%  \item[(a)] $p \le q$.
%  \item[(b)] $n^* \ro{\geq{}} 0$.
%  \item[(c)] In the case $n^*=0$ further restrictions follow:
%  \begin{enumerate}
%  \item[(i)] $r \le q$ and $n=0$;
%  \item[(ii)] If $q = \infty$, then $r \le 1$.
%  \item[(iii)] If $p=q$, then $r \le 2$. 
%  \end{enumerate}
% \end{enumerate}
%\end{theorem}
%
%\begin{proof}
% The proof is largely analogous to the proof of part (a) of Theorem \ref{thm:embed_hom}. Theorem \ref{thm:FV_embed_Besov} notes $p \le q$ and  $a_+^{(q)} \in \ell^{q\cdot (\nicefrac{r}{q})'}(\mathbb{N})$ as necessary condition. Recalling that the $j$th entry of that sequence is  
% \[
%  a_j^{\ro{(q)}} = |{\rm det}(A)|^{\ro{-}jn^*} + |{\rm det}(A)|^{\ro{-}jn^*} \| A^j \|^n~,
% \] and that $A$ is expansive, we find that $n^* \ro{\geq} 0 $ is a necessary condition for $a_+^{(q)}$ to be bounded. In the case of equality we have 
% \[
%  a_j^{\ro{(q)}} = 1  + \| A^j \|^n~,
% \] and then $(a_j)_{\ro{j\in\N_0}} \in \ell^s(\mathbb{N}_{\ro{0}})$ only holds for $s = \infty$, forcing \ro{$n=0$ and} the relation $q\cdot (\nicefrac{r}{q})' = \infty$. \ro{The statements (ii) and (iii) of part (c) now follow} in the same way as for the analogous statements in Theorem \ref{thm:embed_hom}\ro{, using $a_j^{(p)}=2$ when $n=n^{\ast}=0$ and $p=q$}. 
%\end{proof}

\begin{theorem}  \label{thm:embed_inhom_nec_ii}
 Let $A$ denote an expansive matrix. Assume that \[ B_{p,r}^\alpha(A) \hookrightarrow W^{n,q}(\mathbb{R}^d)~
   \] holds. Let 
   \[
    \lambda_{max} = \max \{ |\lambda|~:~ \lambda \mbox{ is an eigenvalue of } A \}~.
   \]
Then the following relations exist between the various parameters:
\begin{enumerate}
 \ro{\item[(a)] $p\leq q$}
 \item[(b)] $n \le \ro{\frac{\ln (|{\rm det}(A)|)}{\ln (\lambda_{max})}  \cdot n^* }$ \ro{(and in particular $n^{\ast}\geq 0$)}.
 \item[(c)] In the case $n = \ro{\frac{\ln (|{\rm det}(A)|)}{\ln (\lambda_{max})} \cdot n^*}$ further restrictions follow:
 \begin{itemize}
  \item[(i)] $r \le q$;
  \item[(ii)]  If $q = \infty$, then $r \le 1$.
  \item[(iii)] If $p=q$, then $r \le 2$. 
  \item[(iv)] If $A$ is not an AND matrix, then $n=n^*=0$.
 \end{itemize}
\end{enumerate}
\end{theorem}

\begin{proof}
\ro{Part (a) is an immediate consequence of Theorem \ref{thm:FV_embed_Besov}.}\\
%\ro{\sout{Let $(a_j^{(q)} )_{j \in \mathbb{N}_{\ro{0}}} = a_+^{(q)}$.}}
  \ro{For the remaining statements,} we have $\|A^j\|\geq\lambda_{\max}^j$ for all $j\in\N_{\ro{0}}$, hence
\begin{equation*}
a_j^{(q)} \geq|\det(A)|^{\ro{-}jn^*}+|\det(A)|^{\ro{-}jn^*}\cdot\lambda_{\max}^{jn}
\end{equation*}
for all $j \ge \ro{0}$. The second term can be rewritten as 
\begin{equation} \label{eqn:sec_term}
 {\lambda_{\max}}^{\ro{-}\log_{\lambda_{\max}}(|\det(A)|)\cdot jn^*}\cdot\lambda_{\max}^{jn} = 
 {\lambda_{\max}}^{j\left(\ro{n-\frac{\ln(|\det(A)|)}{\ln(\lambda_{\max})}\cdot n^*}\right)}
\end{equation}
These terms can only be bounded if the exponents are not positive, i.e. if 
\begin{equation*}
n\leq \ro{\frac{\ln(|\det(A)|)}{\ln(\lambda_{\max})}\cdot n^*}~.
\end{equation*}
holds. This proves \ro{(b)}. 

Assuming \ro{equality, we obtain}
\begin{equation*}
a_j^{(q)} \geq|\det(A)|^{\ro{-}jn^*}+1~,
\end{equation*}
which excludes $s$-summability for any $0<s<\infty$. Just as in the proof of Theorem \ref{thm:embed_hom}, the necessary conditions from Theorem \ref{thm:FV_embed_Besov}(b) now imply the requirements (i) \ro{and (ii). Additionally assuming $p=q$, analogously allows us to estimate $a_j^{(p)}\geq |\det(A)|^{-jn^{\ast}}+1$, which entails $r\leq 2$, again by Theorem \ref{thm:FV_embed_Besov}(b).}

Finally, assume that $n = \ro{\frac{\ln (|{\rm det}(A)|)}{\ln (\lambda_{max})} \cdot n^*}$ and that $A$ is not an AND matrix. We write $A$ as block diagonal matrix with blocks $A_1,A_2$, where $A_1$ is the Jordan block associated to $\lambda_{max}$ and $A_2$ contains the remaining Jordan blocks. By assumption on $A$, \ro{there exists $c>0$}, such that
\[
 \| A^j \| \ro{{}\geq{}} \| A_1^j \| \ge cj \lambda_{\max}^j
\] \ro{for all $j\in\N_0$,} where the additional factor $j$ is supplied by the $j$th power of the superdiagonal part of $A_1$\ro{, when using $\|\cdot\|_{\infty}$ instead of $\|\cdot\|$, and the constant $c$ is obtained by switching back to $\|\cdot\|$.}
In this setting, the assumption $n = \ro{\frac{\ln (|{\rm det}(A)|)}{\ln (\lambda_{max})} \cdot n^*}$ allows to conclude via (\ref{eqn:sec_term}) that 
\[
 a_j^{(q)} \ge |\det(A)|^{\ro{-}jn^*} \| A^j \|^n \ge c^nj^n {\lambda_{\max}}^{j\left(\ro{n-\frac{\ln(|\det(A)|)}{\ln(\lambda_{\max})}\cdot n^*}\right)}  = j^n ~,
\] which is only bounded for $n=0$. This also entails $n^*=0$.
\end{proof}

\ro{As for the sufficient conditions, the following theorem shows that the assumption\linebreak $n<\frac{\ln(|\det(A)|)}{\ln(\lambda_{max})}\cdot n^{\ast}$ is in fact strong enough to guarantee the desired embedding.\\ When $n=\frac{\ln(|\det(A)|)}{\ln(\lambda_{max})}\cdot n^{\ast}$ we will once more have to distinguish between AND and non-AND generators.}

\begin{theorem}
 \label{thm:embed_inhom_suff}
  Let $A$ denote a matrix in expansive Jordan normal form. 
 Assume that the assumption
  \begin{enumerate}
  \item[(a)] $p \le q$;
  \end{enumerate}
  as well as either of the following conditions hold:
  \begin{enumerate}
  \item[(b)] $n <  \ro{\frac{\ln (|{\rm det}(A)|)}{\ln (\lambda_{max})}  \cdot n^*}$.
  \ro{\item[(b')] $n=n^{\ast}=0$, and $r\leq q^{\nabla}$}
  \item[(b'')] $A$ is an AND matrix, $n = \ro{\frac{\ln (|{\rm det}(A)|)}{\ln (\lambda_{max})}  \cdot n^*}$, and $r \le q^\nabla$. 
  \end{enumerate}
  Then the embedding 
  \[ B_{p,r}^\alpha(A) \hookrightarrow W^{n,q}(\mathbb{R}^d)~
   \] holds.
\end{theorem}
\begin{proof}
 By Theorem \ref{thm:FV_embed_Besov}(b), it suffices to check that (b)\ro{, (b') or (b'')} guarantee\linebreak  $a_+^{(q)} \in \ell^{q^{\nabla}\cdot(\nicefrac{r}{q^{\nabla}})^{\prime}}(\mathbb{N}_{\ro{0}})$, where
 \begin{equation} \label{eqn:ajq}
  a_j^{(q)} = |{\rm det}(A)|^{\ro{-}j n^*} + \ro{|\det(A)|^{-jn^{\ast}}\cdot\|A^j\|^n}
 \end{equation}
 
 For the sufficiency of \ro{(b)} pick $\lambda_+>\lambda_{\max}$ sufficiently close to ensure 
\begin{equation*}
n<\frac{\ln(|\det(A)|)}{\ln(\lambda_+)}\cdot n^*~.
\end{equation*}
As a consequence\ro{,} we find $\| A^j \| \le  c \lambda_+^j$ for all $j \ge 0$ with suitable choice of positive constant $c\ro{{}>0}$\ro{, (which is a standard result, proved e.g. in \cite{Bow03})}. Combining this with (\ref{eqn:ajq}) leads to the estimate 
 \begin{equation}\label{eqn:ajq+}
a_j^{(q)} \leq |\det(A)|^{\ro{-}jn^*}+c^n{\lambda_+}^{j\left(\ro{n-\frac{\ln(|\det(A)|)}{\ln(\lambda_+)}\cdot n^*}\right)}~\ro{.}
\end{equation}
 Since $0 \le n < \ro{\frac{\ln (|{\rm det}(A)|)}{\ln (\lambda_{+})}  \cdot n^*}$, we also get $n^*> 0$, i.e. both exponents in (\ref{eqn:ajq}) are strictly negative multiples of $j$. But this implies $s$-summability for any $0 < s \le \infty$, in particular 
 $a_+^{(q)} \in \ell^{q^{\nabla}\cdot(\nicefrac{r}{q^{\nabla}})^{\prime}}(\mathbb{N})$.
 
 \ro{When $n=n^{\ast}=0$, $a_+^{(q)}$ is constant by (\ref{eqn:ajq}), which proves the suffiency of (b'), since $r\leq q^{\nabla}$ implies $q^{\nabla}\cdot\left(\nicefrac{r}{q^{\nabla}}\right)=\infty$.}
 
 Finally, \ro{if $A$ is an AND matrix, to prove the sufficiency of (b''), we can rewrite (\ref{eqn:ajq}) as}
\begin{equation*}
	\ro{a_j^{(q)}=\det(A)|^{jn^*}+c^n{\lambda_+}^{j\left(n-\frac{\ln(|\det(A)|)}{\ln(\lambda_{max})}\cdot n^*\right)}~.}
\end{equation*}
 \ro{Under the additional assumption that $n=\frac{\ln (|{\rm det}(A)|)}{\ln (\lambda_{max})}  \cdot n^*$,} the second term is constant, and the exponent in the first term is a nonpositive multiple of $j$, in particular bounded. This implies $a_+^{(q)} \in \ell^\infty(\mathbb{N}) =  \ell^{q^{\nabla}\cdot(\nicefrac{r}{q^{\nabla}})^{\prime}}(\mathbb{N})$, whenever $r \le q^\nabla$. 
\end{proof}

The following corollary identifies a large class of cases for which the characterization of the embedding is sharp; in addition to the already noted cases $q\in(0,2]\cup\{\infty\}$.

\begin{corollary}
 Assume that \ro{either of the following conditions hold:}
  \begin{enumerate}
  \ro{\item[a)] $\frac{\ln (|{\rm det}(A)|)}{\ln (\lambda_{max})}  \cdot n^*\notin\N_0$};
  \ro{\item[b)] $A$ is not AND and $n^{\ast}\neq 0$.}
  \end{enumerate}
  Then
  \[ B_{p,r}^\alpha(A) \hookrightarrow W^{n,q}(\mathbb{R}^d)~
   \]  is equivalent to the conditions 
   \begin{enumerate}
    \item $p \le q$;
    \item $n < \ro{\frac{\ln (|{\rm det}(A)|)}{\ln (\lambda_{max})}  \cdot n^*}$.
   \end{enumerate}

\end{corollary}

We next give a remark \ro{summarizing the different behaviour of AND and non-AND matrices.}
\begin{remark}
 Let expansive matrices $A,B$ be given with the same eigenvalues, where $A$ is an AND matrix, and $B$ is not. In the interest of briefness, we call an embedding statement
 \[
  B_{p,r}^\alpha(A) \hookrightarrow W^{n,q}(\mathbb{R}^d)
 \] \textbf{decidable}, for a particular choice of parameters $p,r,q,\alpha,n$ if the criteria derived in this paper provide a definite answer. We can then make the following observations:
 \begin{itemize}
  \item If $B_{p,r}^\alpha(A) \hookrightarrow W^{n,q}(\mathbb{R}^d)$ is decidable, then also the statement $B_{p,r}^\alpha(B) \hookrightarrow W^{n,q}(\mathbb{R}^d)$. 
   \item If $B_{p,r}^\alpha(A) \hookrightarrow W^{n,q}(\mathbb{R}^d)$ is decidable and false, the same follows for $B_{p,r}^\alpha(B) \hookrightarrow W^{n,q}(\mathbb{R}^d)$.
  \item If $B_{p,r}^\alpha(A) \hookrightarrow W^{n,q}(\mathbb{R}^d)$ is not decidable, then $B_{p,r}^\alpha(B) \hookrightarrow W^{n,q}(\mathbb{R}^d)$ is either false or not decidable.
  \end{itemize}
We now give a concise example showing that each of the \ro{remaining combinations} may occur. Let
\begin{equation*}
A=\begin{pmatrix} \sqrt{2}&0\\0&\sqrt{2}\end{pmatrix}\quad\text\quad B=\begin{pmatrix} \sqrt{2}&1\\0&\sqrt{2}\end{pmatrix}.
\end{equation*}
We fix $p:=2$, $q:=3>p$ and $\alpha_1:=\frac{5}{3}\geq\frac{1}{6}=\frac{1}{2}-\frac{1}{3}=\frac{1}{p}-\frac{1}{q}=:\alpha_{2}$. Now $\frac{1}{3}+\frac{2}{3}=1$ implies $q^{\nabla}=q^{\prime}=\frac{3}{2}$.

In addition we have 
\begin{equation*}
\frac{\ln(|\det(A)|)}{\ln(\lambda_{\max})}\cdot\left(\alpha_1+\frac{1}{q}-\frac{1}{p}\right) = 2\cdot\left(\frac{5}{3}+\frac{1}{3}-\frac{1}{2}\right)=3.
\end{equation*}
We now apply the criteria for different values of $n$ and $r$:
\begin{enumerate}
\item[(a)] Let $n\in\{0,1,2\}$. Then our criteria yield $B_{p,r}^{\alpha_1}(A)\hookrightarrow W^{n,q}(\R^d)$ \ro{a}nd\linebreak $B_{p,r}^{\alpha_1}(B)\hookrightarrow W^{n,q}(\R^d)$ for all $r>0$.
\item[(b)] Let $n\in\N$ with $n>3$. Then our criteria yield $B_{p,r}^{\alpha_1}(A)\not\hookrightarrow W^{n,q}(\R^d)$ and\linebreak $B_{p,r}^{\alpha_1}(B)\not\hookrightarrow W^{n,q}(\R^d)$, for all $r>0$.
\item[(c)] Fixing $n=3$ and $r=1<\frac{3}{2}$ leads to $B_{p,r}^{\alpha_1}(A)\hookrightarrow W^{n,q}(\R^d)$ and \linebreak$B_{p,r}^{\alpha_1}(B)\not\hookrightarrow W^{n,q}(\R^d)$
\item[(d)] Let $n=3$ and $r=2\in\left(\frac{3}{2},3]\right)$. Then $B_{p,r}^{\alpha_1}(A)\hookrightarrow W^{n,q}(\R^d)$ is not decidable, on the other hand $B_{p,r}^{\alpha_1}(A)\not\hookrightarrow W^{n,q}(\R^d)$ is true and decidable. 
\item[(e)] In the case of $n=0$ and $r=2\in\left(\frac{3}{2},3]\right)$ neither $B_{p,r}^{\alpha_2}(A)\hookrightarrow W^{n,q}(\R^d)$ nor\linebreak $B_{p,r}^{\alpha_2}(B)\hookrightarrow W^{n,q}(\R^d)$ are decidable.
\end{enumerate}
\end{remark}

\begin{remark} \label{rem:summary}
 Let us now briefly summarize the impact of the matrix $A$ on the embedding statements
   \[ B_{p,r}^\alpha(A) \hookrightarrow W^{n,q}(\mathbb{R}^d)~,
   \] 
where we assume $A$ in expansive Jordan normal form. We focus on the interpretation of $\alpha \in \mathbb{R}$ as a smoothness parameter.
 With the remaining parameters $n,p,q,r$ and $A$ fixed, and under the condition $p\le q$, the decisive inequality 
 \begin{equation} \label{eqn:deg_smooth}
  n \le \frac{\ln (|{\rm det}(A)|)}{\ln (\lambda_{max})}  \cdot \left( \alpha + \frac{1}{q} - \frac{1}{p} \right) 
 \end{equation}
 shows that increasing $\alpha$ by one allows to increase the smoothness parameter $n$ by 
\[
 a = \frac{\ln (|{\rm det}(A)|)}{\ln (\lambda_{max})} \in (1,d]~.
\] This quantity can be interpreted as a \textbf{degree of isotropy}, with the extreme values $1$ and $d$ corresponding to the anisotropic and isotropic ends of the scale. 

These considerations first of all show that the interpretation of $\alpha$ as a smoothness parameter is justified in the setting of inhomogeneous anisotropic Besov spaces as well, and that the gain in smoothness is quantified precisely by the degree of isotropy. 

We furthermore observe that the degree of isotropy only depends on the eigenvalues of $A$. The Jordan block structure of $A$ only has an influence on the AND property, which is of a rather secondary nature: It occurs only as part of additional conditions needed to decide the embedding statement in the exceptional cases when equality holds in (\ref{eqn:deg_smooth}).  Note also that only the Jordan block corresponding to the largest eigenvalue is relevant for this property. 

We finally return to the comparison of the embedding behaviours of homogeneous and inhomogeneous Besov spaces. Recall that the former shows no dependence on the expansive matrix whatsoever, whereas the inhomogeneous spaces exhibit a distinctly more nuanced embedding behaviour. This is in marked contrast to the diversity of the Besov scales themselves: By Corollary 6.5 of \cite{MR4135424}, two expansive matrices $A,B$ that induce the same scales of homogeneous Besov spaces also induce the same scale of inhomogeneous spaces; and Remark 7.4 of the cited paper provides a concrete example that the converse is not true. Thus the scales of homogeneous Besov spaces are generally more diverse than the inhomogeneous ones, but the embedding behaviour does not reflect this. 
\end{remark}

\section*{Concluding remarks}
The main purpose of this paper was to elucidate the embedding behaviour of anisotropic Besov spaces into isotropic Sobolev spaces. While this is obviously a legitimate question in its own right, the relative ease with which it could be translated to rather elementary problems in numerical linear algebra demonstrates the power of the general embedding results in \cite{FV}. The generality and scope of these results is further emphasized by the fact that a similar analysis has recently been performed for a completely different class of Besov-type function spaces, the so-called \textit{shearlet coorbit spaces}; see \cite{FuKo_Embed}.

\bibliography{embedding_ai.bib}

\def\cprime{$'$}
\begin{thebibliography}{10}

\bibitem{Bow03}
Marcin Bownik.
\newblock Anisotropic {H}ardy spaces and wavelets.
\newblock {\em Mem. Amer. Math. Soc.}, 164(781):vi+122, 2003.

\bibitem{Bow05}
Marcin Bownik.
\newblock Atomic and molecular decompositions of anisotropic {B}esov spaces.
\newblock {\em Math. Z.}, 250(3):539--571, 2005.

\bibitem{MR4135424}
Jahangir Cheshmavar and Hartmut F\"{u}hr.
\newblock A classification of anisotropic {B}esov spaces.
\newblock {\em Appl. Comput. Harmon. Anal.}, 49(3):863--896, 2020.

\bibitem{FuKo_Embed}
Hartmut Führ and René Koch.
\newblock Embeddings of shearlet coorbit spaces into sobolev spaces.
\newblock {\em International Journal of Wavelets, Multiresolution and
  Information Processing}, 0(0):2040003, 0.

\bibitem{MR1884234}
Reinhard Hochmuth.
\newblock Wavelet characterizations for anisotropic {B}esov spaces.
\newblock {\em Appl. Comput. Harmon. Anal.}, 12(2):179--208, 2002.

\bibitem{MR374900}
Jaak Peetre.
\newblock A remark on {S}obolev spaces. {T}he case {$0<p<1$}.
\newblock {\em J. Approximation Theory}, 13:218--228, 1975.

\bibitem{MR2400793}
F.~J. P\'{e}rez~L\'{a}zaro.
\newblock Embeddings for anisotropic {B}esov spaces.
\newblock {\em Acta Math. Hungar.}, 119(1-2):25--40, 2008.

\bibitem{Ru_FA}
Walter Rudin.
\newblock {\em Functional analysis}.
\newblock International Series in Pure and Applied Mathematics. McGraw-Hill,
  Inc., New York, second edition, 1991.

\bibitem{Triebel_FSI}
Hans Triebel.
\newblock {\em Theory of function spaces}, volume~78 of {\em Monographs in
  Mathematics}.
\newblock Birkh\"auser Verlag, Basel, 1983.

\bibitem{FV}
Felix Voigtlaender.
\newblock Embeddings of decomposition spaces into {S}obolev and {BV} spaces,
  2016.

\end{thebibliography}
\bibliographystyle{plain}
\end{document}